\documentclass[11pt, draft]{amsart}
\usepackage{amssymb, amstext, amscd, amsmath, amssymb}
\usepackage{mathtools, color, paralist, dsfont, rotating}
\usepackage{verbatim}
\usepackage{enumerate}
\usepackage{tikz-cd}
\usepackage{tikz}
\usepackage[all]{xy}
\numberwithin{equation}{section}

\makeatletter
\def\@cite#1#2{{\m@th\upshape\bfseries%
[{#1\if@tempswa{\m@th\upshape\mdseries, #2}\fi}]}}
\makeatother
\theoremstyle{plain}
\newtheorem{theorem}{Theorem}[section]
\newtheorem{corollary}[theorem]{Corollary}
\newtheorem{proposition}[theorem]{Proposition}
\newtheorem{lemma}[theorem]{Lemma}
\theoremstyle{definition}
\newtheorem{definition}[theorem]{Definition}

\newtheorem{remark}[theorem]{Remark}

\newtheorem*{acknow}{Acknowledgements}
\theoremstyle{remark}

\mathtoolsset{centercolon}
\usepackage[a4paper]{geometry}
\newcommand{\A}{{\mathcal{A}}}
\newcommand{\B}{{\mathcal{B}}}
\newcommand{\C}{{\mathcal{C}}}
\newcommand{\D}{{\mathcal{D}}}

\newcommand{\G}{{\mathcal{G}}}
\renewcommand{\H}{{\mathcal{H}}}

\newcommand{\K}{{\mathcal{K}}}

\renewcommand{\O}{{\mathcal{O}}}

\newcommand{\T}{{\mathcal{T}}}

\newcommand{\bN}{\mathbb{N}}

\newcommand{\bZ}{\mathbb{Z}}

\newcommand{\bbC}{{\mathbb{C}}}

\newcommand{\bbI}{{\mathbb{I}}}

\newcommand{\bbN}{{\mathbb{N}}}

\newcommand{\bbZ}{{\mathbb{Z}}}

\newcommand{\fK}{{\mathfrak{K}}}

\newcommand{\foral}{\text{ for all }}

\newcommand{\sot}{\textsc{sot}}

\newcommand{\cenv}{\mathrm{C}_e^*}

\newcommand{\Aut}{\operatorname{Aut}}

\newcommand{\dg}{\operatorname{diag}}

\newcommand{\id}{{\operatorname{id}}}

\newcommand{\oM}{\operatorname{M}}
\newcommand{\oZ}{\operatorname{Z}}

\newcommand{\proj}{\operatorname{proj}}

\newcommand{\spn}{\operatorname{span}}

\newcommand{\ca}{\mathrm{C}^*}

\newcommand{\Ad}{\operatorname{Ad}}

\newcommand{\hphi}{\hat{\phi}}

\newcommand{\sca}[1]{\left\langle#1\right\rangle} 

\begin{document}

\title{Stable isomorphisms of operator algebras}

\author[E.T.A. Kakariadis]{Evgenios T.A. Kakariadis}
\address{School of Mathematics, Statistics and Physics\\ Newcastle University\\ Newcastle upon Tyne\\ NE1 7RU\\ UK}
\email{evgenios.kakariadis@ncl.ac.uk}

\author[E.G. Katsoulis]{Elias~G.~Katsoulis}
\address{Department of Mathematics\\ East Carolina University\\ Greenville, NC 27858\\USA}
\email{katsoulise@ecu.edu}

\author[X. Li]{Xin Li}
\address{School of Mathematics and Statistics\\ University of Glasgow\\ Glasgow, G12 8QQ, Scotland\\ United Kingdom}
\email{Xin.Li@glasgow.ac.uk}



\maketitle

\begin{abstract} Let $\A$ and $\B$ be operator algebras with $c_0$-isomorphic diagonals and let $\K$ denote the compact operators. We show that if $\A\otimes\K$ and $\B\otimes\K$ are isometrically isomorphic, then $\A$ and $\B$ are isometrically isomorphic. If the algebras $\A$ and $\B$ satisfy an extra analyticity condition a similar result holds with $\K$ being replaced by any operator algebra containing the compact operators. For non-selfadjoint graph algebras this implies that the graph is a complete invariant for various types of isomorphisms, including stable isomorphisms, thus strengthening a recent result of Dor-On, Eilers and Geffen. Similar results are proven for algebras whose diagonals satisfy cancellation and have $K_0$-groups isomorphic to $\bbZ$. This has implications in the study of stable isomorphisms between various semicrossed products.
\end{abstract}

\section{Introduction}
There are two lines of inquiry that motivate the present work. Initial motivation comes from recent results of Dor-On, Eilers and Geffen \cite{DEG} that address the hierarchy of various types of isomorphisms between (selfadjoint and non-selfadjoint) operator algebras, with an eye on non-selfadjoint graph algebras and their stable isomorphisms (a previously intractable problem, as the authors comment in the introduction of their paper). In \cite[Theorem 6.4]{DEG} they show that for row-finite graphs, a stable isomorphism between their tensor algebras implies that the corresponding graphs are isomorphic. The proof of this result involves significant $K$-theoretic considerations, an earlier result of the second-named author and Kribs~\cite{KKribs} and relies strongly on the row-finiteness of the graphs involved. Indeed, in \cite[Example 6.5]{DEG}, the authors show that the selfadjoint considerations in their proof of \cite[Theorem 6.4]{DEG} are not valid for graphs which are not row-finite, thus rendering obsolete the technique of their proof for such graphs. Nevertheless this does not exclude the possibility that the ``non-selfadjoint" statement of  \cite[Theorem 6.4]{DEG} is indeed valid with a different proof, thus raising the question whether a stable isomorphism between tensor algebras of arbitrary graphs implies that the graphs are isomorphic.

The other source of motivation originates in the work of the second-named author and Ramsey on non-selfadjoint crossed products~\cite{KR16}. Let $X$ be a locally compact Hausdorff space and let $\sigma :X \rightarrow X$ be a homeomorphism. Let $\K$ denote the compact operators on a separable Hilbert space and let $\K^+$ denote the upper triangular compact operators with respect to a $\bbZ$-ordered orthonormal basis. One can form now a non-selfadjoint crossed product algebra 
\begin{equation} \label{eq;KRcp}
(C(X) \otimes \K^+) \rtimes_{\sigma \otimes \Ad \lambda}\bbZ,
\end{equation}
where $\lambda$ denotes the left regular representation of $\bbZ$. This natural class of non-selfadjoint crossed products begs to be classified and in \cite{KR16} the first step was taken by showing that 
\begin{equation}\label{eq;KR16}
(C(X)\otimes \K^+) \rtimes_{\sigma \otimes \Ad \lambda}\bbZ \simeq (C(X) \otimes_{\sigma} \bbZ^+)\otimes \K,
\end{equation}
where $C(X) \otimes_{\sigma} \bbZ^+$ denotes the semicrossed product of Arveson~\cite{Arv67} and Peters~\cite{Pet}. (The equation \eqref{eq;KR16} follows from the last line in the proof of \cite[Theorem 2.12]{KR16}.) Therefore the classification of the crossed products in (\ref {eq;KRcp}) becomes a problem in the classification theory of semicrossed products up to stable isomorphism. This brings a new perspective to the classification problem for stable isomorphisms between operator algebras and reinforces its study.

Both of the above lines of inquiry are accommodated in this note. First we show that the non-selfadjoint part of Theorem~6.4 of Dor-On, Eilers and Geffen \cite{DEG} is indeed valid beyond row-finite graphs. This comes as a corollary of a more general result that shows that if $\A$ and $\B$ are operator algebras with diagonals isomorphic to $c_0$ so that $\A\otimes\K$ and $\B\otimes\K$ are isometrically isomorphic, then $\A$ and $\B$ are isometrically isomorphic. The proof is elementary and does not require the use of $K$-theory\footnote{\ As we shall see shortly, $K$-theoretic considerations do enter in our study of stable isomorphisms, thus vindicating the intuition of \cite{DEG}.}. If the algebras $\A$ and $\B$ satisfy an extra analyticity condition (Definition~\ref{defn;analytic}), a similar result holds with $\K$ being replaced by any operator algebra containing the compact operators. The proof of this result is more involved and we consider it as the central result of in this line of research.

We also address the stable isomorphism problem for semicrossed products, with an eye on to the classification of the crossed products in (\ref{eq;KRcp}). In Theorem~\ref{thm;smain} we consider unital operator algebras $\A$ and $\B$ whose diagonals have cancellation, their $K_0$-groups are isomorphic to $\bbZ$ and their units belong to the same $K_0$-class, as elements of the diagonal. For such algebras we show that $\A\otimes\K$ and $\B\otimes\K$ are completely isometrically isomorphic if and only if $\A$ and $\B$ are completely isometrically isomorphic. If $X$ is a contractible compact Hausdorff space then $C(X)$ satisfies cancellation, $K_0(C(X))\simeq \bbZ$ and $1\in C(X)$ belongs to the class of the positive generator of $K_0(C(X))$. This allows us to use Theorem~\ref{thm;smain} in order to classify the crossed products of \eqref{eq;KRcp}, provided that $X$ is a contractible, compact Hausdorff space (see Corollary~\ref{motivation2}). Actually, for algebras whose diagonal is of the form $C(X)$, with $X$ contractible, compact Hausdorff space, we obtain a more general result in Theorem~\ref{thm;contract}, whose proof does not involve the use of $K$-theory.

 In this paper we assume that all of our operator algebras, including $\ca$-algebras, are non-degenerately represented on a Hilbert space. In order to avoid set theoretic complications and achieve a smooth presentation, we fix a separable Hilbert space $\H$ and we denote by $\K$ the compact operators acting on $\H$; these will be always referred to as ``\textit{the compact operators}". Furthermore when we say that an operator algebra ``\textit{contains the compact operators}", we mean that the algebra is acting on that separable Hilbert space $\H$ and contains $\K$.

All operator algebras in this paper are approximately unital, i.e., they posses a contractive approximate unit. If $\A$ is an operator algebra, then $\dg\A$ will denote its diagonal, i.e., $\dg\A :=\A\cap \A^*$. Note that the diagonal does not depend on the particular isometric representation of $\A$ and it is therefore well-defined. (See the first paragraph of the proof of Lemma~\ref{wk} below.) If $\A$ happens to be a $\ca$-algebra, then $\oM(\A)$ will denote its multiplier algebra. The symbol $\otimes $ is reserved for the spatial tensor product of operator algebras, i.e., if $\A_i$, $i = 1,2$, are (completely isometrically represented) operator algebras on Hilbert spaces $\H_i$, $i=1,2$, then $\A_1\otimes\A_2$ will denote the norm closed subalgebra of $B(\H_1\otimes\H_2)$ generated by elementary tensors of the form $a_1\otimes a_2$, $a_i\in \A_i$, $i = 1,2$. As with the case of $\ca$-algebras, the spatial tensor product of operator algebras does not depend on the specific (completely isometric) representations used to define it. The closure of $\A_1\otimes\A_2 \subseteq B(\H_1\otimes \H_2)$ in the strong operator topology will be denoted by $\A_1\bar{\otimes}\A_2$.

\section{Algebras with small diagonals}
We start with two results which in one form or another are known. We include proofs for completeness.

\begin{lemma} \label{wk}
Let $\A, \B$ be operator algebras and let $\phi\colon \A \rightarrow \B$ be an isometric isomorphism. Then $\phi(\dg\A)=\dg\B$ and the restriction of $\phi$ on $\dg\A$ is a $*$-isomorphism onto $\dg\B$.
\end{lemma}

\begin{proof}
Recall that in a unital operator algebra $\A$, the unitary elements in the diagonal of $\A$ coincide with the invertible contractions whose inverses are also contractions and contained in  $\A$. Furthermore, these elements generate the diagonal as a vector space.

Let $\widetilde{\A}= \A\oplus \bbC$ and $\widetilde{\B}= \B\oplus \bbC$ be the unitizations of $\A$ and $\B$ respectively that leave the original algebras as  proper ideals in the unitization. Note that $\dg\widetilde{\A} = \widetilde{\dg \A }$ and similarly for $\widetilde{\B}$.

 If $\widetilde{\phi}:\widetilde{\A}\rightarrow \widetilde{\B}$ is the unitization of $\phi$, then Meyer's Theorem \cite[Corollary 2.1.15]{BLM04} shows that $\widetilde{\phi}$ is also an isometry. Hence, the first paragraph of the proof implies that $\widetilde{\phi}$ preserves unitary elements and their adjoints. Therefore $\widetilde{\phi}(\dg \widetilde{\A})=\dg\widetilde{\B}$ and the restriction of $\widetilde{\phi}$ on $\dg\widetilde{\A}$ is a $*$-isomorphism onto $\dg\widetilde{\B}$. The conclusion now follows by restricting $\widetilde{\phi}$ on $\dg\A$.  \end{proof}

\begin{lemma} \label{l;dg}
If $\A$ is an operator algebra, then $$\dg( \A \otimes \K ) =( \dg\A) \otimes \K.$$
\end{lemma}

\begin{proof}
Let $\{ e_i\}_{i \in \bbN}$ be an orthonormal basis for $\H$ and let $e_{ij}$ denote the rank-one operator defined by $e_{ij}(x) = \langle x, e_j\rangle e_i$, $i,j \in \bbN$.

Assume first that $\A$ is unital  and let $a \in \dg\big( \A \otimes \K \big)$. Clearly $1\otimes e_{i j} \in \dg\big( \A \otimes \K \big)$ for all $i, j\in \bbN$ and so 
\[
a_{i j}\otimes e_{i j} \equiv (1\otimes e_{i i} )a(1\otimes e_{j j} ) \in \dg\big( \A \otimes \K \big).
\]
Since $\big(1\otimes \sum_{i=1}^n e_{ii}\big)_{n \in \bbN}$ is an approximate unit for $ \dg\big( \A \otimes \K \big)$, we are to show that $a_{i j} \in \dg \A$, for all $i,j \in \bbN$.

Consider finite families $\{ b_{st}^{(n)}\}_{s t}$, $s, t \in \bbN$, $n \in \bbN$, in $\A$ so that 
\[
a = \lim_n \sum _{s,t} ({b_{st}^{(n)}})^*\otimes e_{st}.
\]
Then, 
\begin{align*}
a_{ij}\otimes e_{ij}&= (1\otimes e_{i i} )a(1\otimes e_{j j} ) \\
            &= \lim_n (1\otimes e_{i i} )\big(\sum _{s, t} ({b_{st}^{(n)}})^*\otimes e_{st}\big)(1\otimes e_{j j} ) \\
            &= \lim_n \, ({b_{ij}^{(n)}})^*\otimes e_{ij}
            \end{align*}
and so $a_{i j} =  \lim_n \, ({b_{ij}^{(n)}})^* \in \A^*$, as desired. 

If $\A$ is not unital, then one replaces $1\in \A$ with an approximate unit for the $\ca$-algebra $\dg \A$ consisting of selfadjoint operators and then repeats the same arguments as above by taking limits.
\end{proof}

If $\A$ is an operator algebra, then $\oZ(\A)$ will denote its center. Recall that if $\C$ is a $\ca$ algebra which is (non-degenerately) represented on a Hilbert space $\H_{\C}$, then we may identify the multiplier algebra $\oM(\C)$ with the \textit{idealizer} of $\C$
\[
\{T\in B(\H_{\C}) \mid Tc\in \C \mbox{ and } cT \in \C, \mbox{ for all } c \in \C\}.
\]
In particular, $\oM(\C) \subseteq \C''$. See Section II.7.3 of \cite{Black} for more details.
The following is a standard result and we include a proof for the convenience of the reader.

\begin{lemma} \label{lem;3inter}
Let $\C$ be a $\ca$-algebra and let $\oM(\C)$ denote its multiplier algebra. Then 
\begin{equation} \label{eq;centerB}
\oZ(\oM(\C))= \oZ(\C'') \cap  \oM(\C).
\end{equation}
\end{lemma}

\begin{proof}
Since $\C \subseteq \oM(\C) \subseteq \C''$ we have that $\oZ(\oM(\C)) \supseteq \oZ(\C'') \cap  \oM(\C)$.
Moreover $\oM(\C)$ is dense in $\C''$ in the strong operator topology, and thus $\oZ(\oM(\C)) \subseteq \oZ(\C'') \cap  \oM(\C)$, as the product is separately continuous with respect to the strong operator topology.
\end{proof}

\begin{proposition} \label{cor;repl}
Let $\C$ and $\fK$ be $\ca$-algebras and assume that $\fK$ contains the compact operators. Then 
\[
\oZ(\oM(\C \otimes \fK) ) =  \oZ(\oM(\C) ) \otimes \bbC I .
\]
\end{proposition}

\begin{proof}
Since $(\C \otimes \fK)' = \C'\otimes \bbC I$ and $(\C \otimes \fK)'' = \C'' \bar{\otimes} B(\H) $, we have 
\begin{equation} \label{eq;repl1}
\oZ( (\C\otimes \fK)'') = (\C \otimes \fK)' \cap (\C \otimes \fK)'' = (\C' \cap\C'' )\otimes \bbC I = \oZ(\C'' )\otimes \bbC I .
\end{equation}
Using a minimal projection $p \in \fK$, one can verify that 
\begin{equation} \label{eq;repl2}
( \oZ(\C'' )\otimes \bbC I ) \cap \oM(\C \otimes \fK) \subseteq (\oZ(\C'') \cap\oM(\C) )\otimes \bbC I .
 \end{equation}
 Indeed, let $S \in \oZ(\C'')$ so that $S\otimes I \in \oM(\C\otimes \fK)$. Then for any $c \in \C$ we have 
 \[
 (S\otimes I)(c \otimes p) , (c \otimes p )(S \otimes I) \in \C \otimes \fK
 \]
 and so
 \[
 Sc\otimes p , cS \otimes p \in (I \otimes p)(\C\otimes \fK)(I \otimes p)=\C \otimes \bbC p,
 \]
 from which we deduce that $Sc, cS \in \C$, i.e., $S \in \oM(\C)$, as desired.
 
 Now by using (\ref{eq;repl1}), we may replace $ \oZ(\C'' )\otimes \bbC I  $ in (\ref{eq;repl2}) with $\oZ( (\C\otimes \fK)'') $ to obtain
 \[
\oZ( (\C\otimes \fK)'')  \cap \oM(\C \otimes \fK) \subseteq (\oZ(\C'') \cap\oM(\C) )\otimes \bbC I ,
 \]
 which with the aid of Lemma~\ref{lem;3inter} transforms to 
 \[
 \oZ(\oM(\C \otimes \fK) ) \subseteq \oZ(\oM(\C) ) \otimes \bbC I .
 \]
 In order to prove the reverse inclusion, note that if $S \in \oZ(\oM(\C) )$, then clearly $S\otimes I \in \oM(\C \otimes \fK) $. Furthermore, if $c\otimes k \in \C \otimes \fK$ then
 \begin{equation} \label{eq;repl3}
 (S\otimes I)(c\otimes k)=(c\otimes k)(S\otimes I),
 \end{equation}
 because $S$ commutes with every element in $\oM(\C)$ and thus with every element in $\C$. Now $\C\otimes \fK$ is strictly dense in $\oM(\C\otimes \fK)$. Therefore (\ref{eq;repl3}) implies that $S \otimes I \in  \oZ(\oM(\C \otimes \fK) ) $ and we are done. 
\end{proof}

For the rest of the paper we make the blanket assumption that all operator algebras are approximately unital with a contractive approximate unit contained in the diagonal. 

We say that an operator algebra $\A$ has a \textit{$c_0$-isomorphic diagonal} if there exists a perhaps finite set $\Gamma$ so that $\dg\A \simeq c_0(\Gamma)$. Due to our blanket assumption, we see that an operator algebra $\A$ with $c_0$-isomorphic diagonal necessarily contains an approximate identity consisting of commuting selfadjoint operators.\footnote{\ Many algebras of current interest, including various tensor algebras of non-degenerate product systems \cite{DK20, DKKLL}, satisfy this requirement.
Further examples include non-selfadjoint algebras with trivial diagonal, e.g., \cite{DRS11, DRS15, SS16, SSS18, SSS20}. As we are about to show, stable isometric isomorphisms implie isometric isomorphisms in these cases.} 
 
 \begin{theorem} \label{thm;main1}
Let $\A$ and $\B$ be operator algebras with $c_0$-isomorphic diagonals. If $\A\otimes\K$ and $\B\otimes\K$ are isometrically isomorphic, then $\A$ and $\B$ are isometrically isomorphic. 
 \end{theorem}
 
 \begin{proof}Assume that $\A\otimes\K$ and $\B\otimes\K$ are isometrically isomorphic via a map 
 \[
 \phi:\A\otimes\K \longrightarrow \B\otimes\K.
 \]
 We are to prove that $\A$ and $\B$ are isomorphic. Let $\{p_i\}_{ i \in \Gamma_{\A}}$ and $\{q_j\}_{j \in\Gamma_{\B}}$ be families of mutually orthogonal projections so that 
\[
\dg\A = \ca(\{p_i\mid i \in\Gamma_{\A}\}) \mbox{ and } \dg\B = \ca(\{q_j\mid j \in \Gamma_{\B}\}).
\]
By our blanket assumption, the nets of finite sums from $\{p_i\}_{i \in\Gamma_{\A}}$ and $\{q_j\}_{j \in \Gamma_{\B}}$ form approximate units for $\A$ and $\B$ respectively. By Lemma~\ref{l;dg}, we have that $\dg(\A\otimes\K) = \ca(\{p_i\}_{i})\otimes \K$ and by Lemma~\ref{wk} we have that the restriction of $\phi$ on $\ca(\{p_i\}_{i})\otimes \K$ is a $*$-isomorphism onto the diagonal $\ca(\{q_j\}_{j})\otimes \K$ of $\B\otimes \K$. This isomorphism extends to a $*$-isomorphism 
\[
\bar{\phi}: \oM \! \big( \ca(\{p_i\}_{i})\otimes \K \big) \longmapsto \oM\!\big( \ca(\{q_j\}_{j})\otimes \K \big)
\]
between the corresponding multiplier algebras. Now by Proposition~\ref{cor;repl} and the fact that $\oM\!\big( \ca(\{p_i\}_{i}) = \ca(\{p_i\}_{i})''$ we have that
\[
\oZ\!\left( \oM\!\big( \ca(\{p_i\}_{i})\otimes \K \big) \right) = \ca(\{p_i\}_{i})'' \otimes \bbC I,
\]
and so $\bar{\phi}$ maps $\{p_i\otimes1\}_{i \in \Gamma_{\A}}$, i.e., the minimal projections in $\oZ\! \left( \oM \! \big( \ca(\{p_i\}_{i})\otimes \K \big) \right) $, onto $\{q_j\otimes1\}_{i \in \Gamma_{\B}}$ since these are the minimal projections in $\oZ \! \left( \oM\! \big( \ca(\{q_j\}_{j})\otimes \K \big) \right) $. Therefore, by relabelling $\Gamma_{\B}$ if necessary, we may assume that $\Gamma_{\A}=\Gamma_{\B} = \bbI$ and 
\begin{equation} \label{eq;fininf}
\bar{\phi}(p_i\otimes I) = q_i\otimes I,  \mbox{ for all } i \in \bbI. 
\end{equation}
By using Lemma~\ref{l;dg}, we obtain
\begin{align*}
\phi(p_i\otimes \K)&= \bar{\phi}(p_i\otimes I) \phi\big(\ca(\{p_n\}_{n=1}^{\infty}) \otimes \K\big) \\
				&= \bar{\phi}(p_i\otimes I) \phi\big(\dg(\A\otimes\K)\big)\\
			&=(q_i\otimes I) \big(\ca(\{q_n\}_{n=1}^{\infty}) \otimes \K \big) \\
				&=q_i\otimes  \K,
\end{align*}
for all $i \in \bbI$. 
This allows us to view the restriction of $\phi$ on $p_i\otimes \K$ as an automorphism of $\K$ and so Corollary~3 of \cite[Theorem 1.4.4]{Arvbook} implies that for every $i \in \bbI$ there exists a unitary $u_i \in B(\H)$ so that $$\phi(p_i\otimes x)=q_i\otimes u_i^*xu_i, \mbox{ for all } x \in \K.$$
Consider the unitary $U:=\sot-\sum_{i \in \bbI}  q_i\otimes u_i$. Note that on a dense subset of $\B \otimes \K$ we have 
\[
U \spn\{ q_i\B q_j\otimes \K\mid i,j \in \bbI\} U^*= \spn\{ q_i\B q_j\otimes \K\mid i,j \in \bbI\},
\]
and so conjugation by $U$ determines an automorphism of $\B \otimes \K$.
Therefore by conjugating $\phi$ with $U$ if necessary, we may assume that 
\[
\phi(p_i\otimes x)=q_i\otimes x, \mbox{ for all } x \in \K.
\]
Let $e \in \K$ be a rank-one projection and let $i, j \in \bbI$. Notice that 
\begin{align*}
\phi(p_i\A p_j\otimes e) &=\phi\big((p_i\otimes e)(\A \otimes \K )(p_j\otimes e)\big)\\
&=(q_i\otimes e )(\B\otimes \K ) (q_j\otimes e) \\
&= q_i\B q_j\otimes e.
\end{align*}
Therefore for each pair $i, j \in \bbI$ and any $a \in \A$, there exists a unique element $\phi_{ij}(p_iap_j ) \in q_i\B q_j$ satisfying
$\phi(p_iap_j \otimes e)=\phi_{ij}(p_iap_j ) \otimes e$, thus obtaining a surjective linear map 
\[
\phi_{ij}\colon p_i\A p_j\longrightarrow q_i\B q_j.
\]
Putting all the $\phi_{ij}$ together, we obtain a map 
\[
\hphi \colon \spn\{ p_i\A p_j\mid i,j \in \bbI\} \longrightarrow  \spn\{ q_i\B q_j\mid i,j \in \bbI\}.
\]
satisfying
\begin{equation} \label{hphi}
\hphi (a)= \sum_{i,j \in F} \phi_{ij}(p_iap_j), 
\end{equation}
provided that $a = \sum_{i,j \in F} p_iap_j \in \A$, $F\subseteq \bbI$ finite. 
Notice that with such an $ a \in \A$ we have  
\begin{equation} \label{hphiisom}
\begin{split}
\|\hphi (a) \|&= \|\sum_{i, j \in F} \phi_{ij}(p_iap_j)\|  = \|\sum_{i, j \in F} \phi_{ij}(p_iap_j) \otimes e\|  \\
&= \|\sum_{i, j \in F} \phi(p_iap_j \otimes e)\| = \|\phi\big( \sum_{i, j \in F} p_iap_j\otimes e \big)\| \\
& = \| \sum_{i, j \in F} p_iap_j\otimes e \| = \| \sum_{i, j \in F} p_iap_j \| \\
&=\|a\|,
\end{split}
\end{equation}
i.e., $\hphi$ is an isometry. Since finite sums from $\{p_i\}_{i \in \bbI}$ and $\{q_j\}_{j \in \bbI}$ form approximate units for $\A$ and $\B$ respectively, $\hphi$ extends to a linear isometry from $\A$ onto $\B$, denoted again as $\hphi$. 

Finally for any $a,a' \in \A$ and $i, j, s, t \in \bbI$, we have, 
\begin{align*}
\hphi(p_iap_j)\hphi(p_sa'p_t)\otimes e &=\big(\hphi(p_iap_j)\otimes e \big)\big(\hphi(p_sa'p_t)\otimes e \big) \\
          &=\big(\phi_{ij}(p_iap_j)\otimes e \big)\big(\phi_{st}(p_sa'p_t)\otimes e \big) \\
         &= \phi(p_iap_j\otimes e) \phi(p_sa'p_t\otimes e )  \\
          &=\phi(p_iap_jp_sa'p_t\otimes e)\\
          &=\hphi (p_iap_jp_sa'p_t)\otimes e,
\end{align*}
which implies that $\hphi$ is multiplicative and therefore $\A$ and $\B$ are isometrically isomorphic via $\hphi$.
\end{proof}

\begin{corollary}
Let $\A$ and $\B$ be operator algebras with $c_0$-isomorphic diagonals. Then $\A\otimes\K$ and $\B\otimes\K$ are completely isometrically isomorphic if and only if $\A$ and $\B$ are completely isometrically isomorphic. 
\end{corollary}

\begin{proof}
If two operator algebras $\A$ and $\B$ are completely isometrically isomorphic, then the same is true for $\A\otimes \K$ and $\B\otimes \K$. 

Conversely assume that $\A\otimes\K$ and $\B\otimes\K$ are completely isometrically isomorphic via a map $\phi:\A\otimes\K \rightarrow \B\otimes\K$. Arguing as in the proof of the previous Theorem, we may assume that $
\phi(p_i\otimes x)=q_i\otimes x$, for all $x \in \K$,  and therefore define a multiplicative map $\hphi : \A \rightarrow \B$ as in (\ref{hphi}). The complete isometricity of $\hphi$ follows from the matricial analogue of (\ref{hphiisom}).
\end{proof}

If $\G$ is a countable directed graph, then $\T^+(\G)$ denotes the tensor algebra of $\G$, i.e., the non-selfadjoint subalgebra of the Toeplitz-Cuntz-Krieger $\ca$-algebra $\T(\G)$ generated by the isometries and projections satisfying the defining relations of $\T(\G)$. (See \cite{KK06b, MS98b} for more information.) The following strengthens \cite[Theorem 6.4]{DEG} by posing no restrictions on the graphs involved.

\begin{corollary} \label{cor;graph}
If $\G$ and $\G'$ are countable directed graphs, then the following are equivalent:
\begin{itemize}
\item[\textup{(i)}] $\G$ and $\G'$ are isomorphic graphs;
\item[\textup{(ii)}] $\T^+(\G)$ and $\T^+(\G')$ are bicontinuously isomorphic;
\item[\textup{(iii)}] $\T^+(\G)\otimes \K$ and $\T^+(\G')\otimes \K$ are (completely) isometrically isomorphic.
\end{itemize} 
\end{corollary}

\begin{proof}
The equivalence of (i) and (ii) follows from \cite[Theorem 2.11]{KK06b}. 

Clearly (i) implies that $\T(\G)$ and $\T(\G')$ are $*$-isomorphic and so $\T(\G)\otimes \K$ and $\T(\G')\otimes \K$ are canonically $*$-isomorphic. From this it follows that the non-selfadjoint algebras $\T^+(\G)\otimes \K$ and $\T^+(\G')\otimes \K$ are (completely) isometrically isomorphic, i.e., (iii) holds. 

Finally, the diagonal of the tensor algebra of a countable graph is isomorphic to $c_0$ and so (iii) implies (ii) by Theorem~\ref{thm;main1}.
\end{proof}

We are interested in strengthening the previous results by replacing the compact operators $\K$ with more general operator algebras. We begin by providing a substitute for Lemma~\ref{l;dg}. For this purpose we focus on a special class of operator algebras.

Let $(G, P)$ be a discrete, ordered abelian group, i.e., $G$ is a discrete abelian group and $P \subseteq G$ is a cone, i.e., $P+P\subseteq P$ and $P\cap (-P)=\{ e\}$, where $ e \in G$ denotes the unit element. Let $\C$ be a $\ca$-algebra and let $\alpha\colon \hat{G}\rightarrow \Aut (\C)$ be a strongly continuous action of the Pontryagin dual of $G$. This allows us to define the Fourier coefficients for any $c \in \C$ by 
\[
\hat{c}(s):=\int_{\hat{G}} \alpha_{\gamma} (c) \overline{\gamma(s)}d\gamma, \,\, s \in G,
\]
where $d\gamma$ is the Haar measure on $\hat{G}$.
In particular the Fourier coefficient corresponding to the unit $e \in G$ determines a faithful expectation 
\[
E_{\alpha} (c ) \colon \C \longrightarrow \C; c \longmapsto \hat{c}(e).
\]
In this way, given any $c \in \C$ we have a formal Fourier series expansion $c \sim \sum_{s \in G} \hat{c}(s)$ which converges to $c$ in a Cesaro-type summation. In particular every element of $\C$ is uniquely determined by its Fourier coefficients.

\begin{definition} \label{defn;analytic}
An operator algebra $\A$ is said to be \textit{analytic} if there exists a discrete, abelian ordered group $(G,P)$ and a strongly continuous action $\alpha\colon \hat{G}\rightarrow \Aut (\ca(\A)) $ of the dual group leaving $\A$ invariant, so that: 
\begin{itemize}
\item[(i)] $\A\subseteq \{ a \in \ca(\A)\mid \hat{a}(s)=0, \mbox{ for all } s \notin P\}$,
\item[(ii)] $\dg\A=  E_{\alpha}(\A).$
\end{itemize}
If $\A$ satisfies only condition (i), then $\A$ is said to be \textit{quasi-analytic}; in that case it is easy to see that we only have $\dg \A \subseteq E_{\alpha}(\A)$.
\end{definition}

The tensor algebra of any product system over an abelian lattice order $(G, P)$ \cite{DK20} is easily seen to be an analytic algebra via the gauge action of $\hat{G}$. 

\begin{lemma} \label{l;dg0}
Let $\A$ be an analytic operator algebra with abelian diagonal and let $\B$ be any operator algebra. Then 
\[
\dg(\A\otimes \B) = (\dg\A) \otimes (\dg\B).
\]
\end{lemma}

\begin{proof}
Let $\alpha\colon \hat{G}\rightarrow \ca(\A)$ be the strongly continuous action of $G$ that determines the analyticity of $\A$ so that $\dg(\A) = E_{\alpha}(\A)$. Now notice that $\A\otimes \B$ is quasi-analytic via the action 
\[
\alpha \otimes \id : \hat{G} \longrightarrow \Aut \ca(\A \otimes \B) 
\]
and so 
\begin{align*}
\dg(\A\otimes \B) &\subseteq E_{\alpha\otimes \id}(\A\otimes \B) \cap E_{\alpha\otimes \id}(\A\otimes \B)^*\\
&= (E_{\alpha} \otimes \id)(\A\otimes \B) \cap (E_{\alpha} \otimes \id)(\A\otimes \B)^*\\
&= ((\dg\A)\otimes \B) \cap  ((\dg\A)\otimes \B)^*.
\end{align*}
Now $\dg\A$ is abelian and so $\dg\A \simeq C_0(X)$ for some locally compact Hausdorff space $X$. Therefore, $$(\dg\A)\otimes \ca(\B)\simeq C_0(X, \ca(\B))$$ and under that identification $(\dg\A)\otimes \B \simeq C_0(X, \B)$ and $((\dg\A)\otimes \B)^* \simeq C_0(X, \B^*)$.  But then under the above identification we have that
\begin{align*}
((\dg\A)\otimes \B) \cap  ((\dg\A)\otimes \B)^*& \simeq C_0(X, \B\cap\B^*) \\
				&\simeq (\dg\A)\otimes (\dg\B)
\end{align*}
and the conclusion follows.
\end{proof}

Note that the proof of Lemma~\ref{l;dg0} establishes something stronger than the isomorphism of two operator algebras. If the algebras $\A$ and $\B$ act on Hilbert spaces $\H_{\A}$ and $\H_{\B}$ respectively, then Lemma~\ref{l;dg0} shows that when we consider $\A\otimes \B$ as a subalgebra of $B(\H_{\A}\otimes \H_{\B})$, then 
\[
(\A\otimes \B)\cap(\A\otimes \B)^*= (\A\cap\A^*)\otimes (\B\cap \B^*)
\]
as sets.

\begin{corollary} \label{l;dgcent}
Let $\A$ be an analytic operator algebra with $c_0$-isomorphic diagonal and let $\fK$ be any operator algebra containing the compact operators. Then 
\[
\oZ\!\big(\oM (\dg( \A \otimes \fK))\big) =(\dg\A)'' \otimes \bbC I .
\]
\end{corollary}

\begin{proof} By Lemma~\ref{l;dg0}, we have that $\dg( \A \otimes \fK) =( \dg\A) \otimes (\dg\fK)$. Since $ \dg \fK$ contains the compact operators as well, Proposition~\ref{cor;repl} implies that  
\[
\oZ\!\big(\oM(( \dg\A) \otimes (\dg\fK) ) \big) = \oZ(\oM(\dg\A) )\otimes \bbC I
\]
and so 
\[
\oZ(\oM(\dg( \A \otimes \fK) )) = \oZ(\oM(\dg\A) )\otimes \bbC I .
\]
 If $\dg \A = \ca(\{p_i\mid i \in\bbN \}) $ for some family $\{p_i \}_{i \in \bbN}$ of mutually orthogonal projections spanning the identity, then it is easy to see that $\oM(\dg\A)  =(\dg \A)''$ and the conclusion follows.
\end{proof}

\begin{lemma} \label{lm;basic}
Let $h_1,h_2,\dots,h_n$ be mutually orthogonal unit vectors in $\H$ and let \break$f_1, f_2, \dots, f_n$ be unit vectors in $\H$ satisfying $\sum_{i\neq j} |\langle f_i,f_j\rangle| <1$. Consider the linear operator
\begin{equation} \label{eq;operator}
T : \spn\{h_1,h_2,\dots, h_n\} \longrightarrow \spn\{f_1, f_2, \dots, f_n\}
\end{equation}
defined by $Th_i=f_i$, $i=1, 2, \dots , n$.
Then,
\begin{equation} \label{eq;bi}
\big(1-\sum_{i\neq j} |\langle f_i,f_j\rangle| \big)^{1/2}\|h\|\leq\|Th\|\leq\big(1+ \sum_{i \neq j} |\langle f_i,f_j\rangle| \big)^{1/2}\|h\|,
\end{equation}
for any $h \in  \spn\{h_1,h_2,\dots, h_n\} $.
\end{lemma}

\begin{proof}
Since in general $\||T| h\| = \|Th\|$, it is enough to prove the result for $|T|$ instead of $T$. 

Notice that for any $1\leq i,j\leq n$, we have
\[
\langle |T|^2 h_i, h_j\rangle  = \langle Th_i, Th_j\rangle =\langle f_i, f_j\rangle 
\]
and so the matrix of $|T|^2$ with respect to the basis $h_1,h_2,\dots,h_n$ is $( \langle f_j, f_i \rangle)_{ij}$. Hence 
\begin{equation} \label{eq;sum1}
\| |T|\|^2=\| |T|^2\|\leq 1+\| |T|^2 - I \|\leq 1 +\sum_{i \neq j} |\langle f_i,f_j\rangle| .
\end{equation}
On the other hand 
\[
\| I - |T|^2\|\leq \sum_{i \neq j} |\langle f_i,f_j\rangle| <1,
\]
and so $|T|^2$ is invertible and 
\begin{equation} \label{eq;sum2}
\| |T|^{-1} \|^2= \| |T|^{-2}\| \leq \big( 1- \sum_{i \neq j} |\langle f_i,f_j\rangle| \big)^{-1}.
\end{equation}
Putting (\ref{eq;sum1}) and (\ref{eq;sum2}) together we have 
\[
\big(1-\sum_{i\neq j} |\langle f_i,f_j\rangle| \big)^{1/2}\|h\|\leq\||T|h\|\leq\big(1+ \sum_{i \neq j} |\langle f_i,f_j\rangle| \big)^{1/2}\|h\|,
\]
as desired.
\end{proof}

We have arrived to the main result of this section.

 \begin{theorem} \label{thm;main2}
Let $\A$ and $\B$ be analytic operator algebras with $c_0$-isomorphic, separable diagonals and let $\fK_{\A}$ and $\fK_{\B}$ be operator algebras containing the compact operators. If $\A \otimes\fK_{\A}$ and $\B  \otimes\fK_{\B}$ are isometrically isomorphic, then $\dg \fK_{\A}$ and $\dg \fK_{\B}$ are  $*$-isomorphic and the algebras $\A$ and $\B$ are bicontinuously isomorphic. 
 \end{theorem}
 
 \begin{proof}
Assume that \[
\phi:\A\otimes\fK_{\A} \longrightarrow \B\otimes\fK_{\B}
\] 
is an isometric isomorphism.
Let $\{p_i\}_{ i \in \bbN}$ and $\{q_j\}_{j \in\bbN}$ be families of mutually orthogonal projections\footnote{\ We allow for the possibility that some of these projections may be zero in order to cover the case where the diagonal is finite dimensional.} so that 
\begin{equation*} 
\dg\A = \ca(\{p_i\mid i  \in \bbN\}) \mbox{ and } \dg \B = \ca(\{q_i \mid i \in \bbN\}).
\end{equation*}
By Lemma~\ref{wk} we have that $\phi$ preserves diagonals and so it extends to a $*$-isomorphism 
\[
\bar{\phi}\colon \oM \! \big(\!\dg(\A\otimes\fK_{\A})\big) \longrightarrow \oM\! \big(\!\dg(\B\otimes\fK_{\B})\big).
\]
Since $\bar{\phi}$ is an isomorphism, it maps the minimal projections of the center of $\oM \big(\!\dg(\A\otimes\fK_{\A})\big)$ onto the minimal projections of the center $\oM\big(\!\dg(\B\otimes\fK_{\B})\big)$. Corollary~\ref{l;dgcent} implies now that, perhaps after relabeling, we have 
\begin{equation} \label{eq;fininf}
\bar{\phi}(p_i\otimes I) = q_i\otimes I,  \mbox{ for all } i \in \bbN. 
\end{equation}
By using Lemma~\ref{l;dg0}, we obtain
\begin{align*}
\phi(p_i\otimes \dg\fK_{\A} )&= \bar{\phi}(p_i\otimes I)) \phi\big(\ca(\{p_n\}_{n=1}^{\infty}) \otimes \dg \fK_{\A} \big) \\
				&= \bar{\phi}(p_i\otimes I)) \phi\big(\dg(\A\otimes \fK_{\A})\big)\\
			&=(q_i\otimes I) \big(\ca(\{q_n\}_{n=1}^{\infty}) \otimes \dg \fK_{\B} \big) \\
				&=q_i\otimes \dg \fK_{\B},
\end{align*}
for all $i \in \bbN$. Hence for each $i \in \bbN$ we may view the restriction of $\phi$ on $p_i\otimes \dg \fK_{\A}$ as an isometric homomorphism, hence a $*$-isomorphism from $ \dg \fK_{\A}$ onto $\dg \fK_{\B}$. Furthermore such an isomorphism is spatially (actually unitarily) implemented by \cite[Theorem 17.7]{Dav} and so the restriction of $\phi$ on $p_i\otimes \dg \fK_{\A}$ maps rank-one operators to rank-one operators.
\vspace{.1in}

Assume now that $\dg \A$ is infinite dimensional and so $\dg \B$ is infinite dimensional as well by (\ref{eq;fininf}).

\vspace{.1in}

\noindent{\textbf{Claim 1.}} \textit{There exists an orthonormal set $\{e_i\}_{i=1}^{\infty}$ and a collection $\{f_i\}_{i=1}^{\infty}$ of unit vectors in $\H$ so that $\phi(p_i\otimes e_{ii})=q_i\otimes f_{ii}$, $i \in \bbN$ and 
\begin{equation} \label{eq;sum}
\sum_{i\neq j} |\langle f_i,f_j\rangle | \leq 1/4.
\end{equation}
}(Here, for vectors $g_i, g_j \in \H$, we write $g_{ij}$ for the rank one operator $g_{ij}(x)=\langle x, g_j\rangle g_i$, $x \in \H$.)

 \vspace{.1in}

\textit{Proof of Claim 1.} Both $\{e_i\}_{i=1}^{\infty}$ and $\{f_i\}_{i=1}^{\infty}$ will be constructed concurrently via induction. Let $e_1\in \H$ be any unit vector. Since $\phi$ is a $*$-isomorphism that preserves rank one operators, we may choose a unit vector $f_1 \in \H$ so that 
\[\phi(p_1\otimes e_{11} ) = q_1\otimes f_{11}.
\]
Extend the singleton $\{ e_1\}$ to an orthonormal set $\{e_1, h_2, h_3, \dots \} \subseteq \H$. Consider vectors $g_2,g_3, \dots$ such that $\phi(p_2 \otimes h_{ii}) = q_2 \otimes g_{ii}$, for all $i \geq 2$.
Since the projections $\{p_2 \otimes h_{ii}\}_{i = 2}^{\infty}$ are mutually orthogonal, the same is true for the projections $\{q_2 \otimes g_{ii}\}_{i = 2}^{\infty}$ and so the set $\{g_i\}_{i= 2}^{\infty}$ is orthonormal. This implies that $\lim_i |\sca{f_1, g_i}| = 0$.
Therefore there exists index $i_2\geq 2$ such that $|\sca{f_1, g_{i_2}}| \leq 2^{-6}$.
Set $e_2 := h_{i_2}$, $f_2 := g_{i_2}$ and notice that so far we have
\[
\phi(p_1 \otimes e_{11}) = q_1 \otimes f_{11},
\phi(p_2 \otimes e_{22}) = q_2 \otimes e_{22}
\text{ and }
|\sca{f_1, f_2}| \leq \frac{1}{2 \cdot 2^{2+3}}.
\]

Suppose now that we have constructed $\{e_1, \dots, e_n\}$ and $\{f_1, \dots, f_n\}$ so that
\[
\phi(p_i \otimes e_{ii}) = q_i \otimes f_{ii}, i =1, \dots, n,
\]
and furthermore satisfying
\[
|\sca{f_i, f_j}| \leq \frac{1}{j \cdot 2^{j+3}}, \foral i < j \leq n.
\]
Extend $\{e_1, \dots, e_n\}$ to an orthonormal set $\{e_1, \dots, e_n, h'_{n+1}, h'_{n+2} \dots \} \subseteq \H$ and consider vectors $g'_{n+1}, g'_{n+2}, \dots$ so that
\[
\phi(p_{n+1} \otimes h'_{ii}) = q_{n+1} \otimes g'_{ii},
\foral i \geq n+1.
\]
Once again the set $\{g'_i\}_{i = n+1}^{\infty} $ is orthonormal. Thus
 $\lim_i |\sca{f_k, g'_i}| = 0$, $k=1, \dots, n$, and so we can choose index ${i_{n+1}}$ so that
 \[
|\sca{f_k, g_{i_{n+1}}}| \leq ((n+1) \cdot 2^{n+4})^{-1}, \mbox{ for all } k=1, \dots, n.
\]
Set $e_{n+1} := h'_{i_{n+1}}$, $f_{n+1} := g'_{i_{n+1}}$ and note that
\[
\phi(p_{n+1} \otimes e_{n+1 \, n+1}) = q_{n+1} \otimes f_{n+1 \, n+1}
\]
and
\[
|\sca{f_k, f_{n+1}}| \leq \frac{1}{(n+1) \cdot 2^{(n+1) + 3}},
\foral
k=1, \dots, n,
\]
which by the inductive hypothesis implies that
\[
|\sca{f_i, f_{j}}| \leq \frac{1}{j \cdot 2^{j + 3}},
\foral
i < j \leq n+1.
\]
This completes the inductive step.

We have obtained an orthonormal set $\{e_i\}_{i =1}^\infty$ and a collection $\{f_i\}_{i=1}^\infty$ of unit vectors such that
\[
\phi(p_i \otimes e_{ii}) = q_i \otimes f_{ii}
\foral i\in \bN,
\text{ and }
|\sca{f_i,f_j}| \leq \frac{1}{j \cdot 2^{j+3}},
\foral
i < j.
\]
It remains to show that $\sum_{i \neq j} |\sca{f_i, f_j}| \leq 1/4$.
For $N \in \bN$ we have that
\begin{align*}
\sum_{i \neq j; i,j \leq N} |\sca{f_i, f_j}|
=
2 \sum_{j=1}^N \sum_{i=1}^j |\sca{f_i, f_j}|
\leq
2 \sum_{j=1}^N \sum_{i=1}^j \frac{1}{j \cdot 2^{j+3}}
=
\frac{1}{4} \sum_{j=1}^N \frac{1}{2^j}.
\end{align*}
Letting $N \to \infty$ we obtain
\[
\sum_{i \neq j} |\sca{f_i,f_j}| \leq \frac{1}{4} \sum_{j=1}^\infty \frac{1}{2^j} = \frac{1}{4},
\]
as required.
This completes the proof of the claim.
 
 \vspace{.1in}
 
Now, we have 
\begin{align*}
\phi(p_i\A p_j\otimes e_{ij}) &=\phi\big((p_i\otimes e_{ii})(\A \otimes \fK_{\A} )(p_j\otimes e_{jj})\big)\\
&=(q_i\otimes f_{ii} )(\B\otimes \fK_{\B} ) (q_j\otimes f_{jj}) \\
&= q_i\B q_j\otimes f_{ij}
\end{align*}
and so for each pair $i, j \in \bbN$ we have a surjective linear map 
\[
\phi_{ij}\colon p_i\A p_j\longrightarrow q_i\B q_j
\]
satisfying $\phi(p_iap_j \otimes e_{ij}) =\phi_{ij}(p_iap_j ) \otimes f_{ij}$, for any $a \in \A$. By putting all the $\phi_{ij}$ together, we obtain once again a map 
\[
\hphi \colon \spn\{ p_i\A p_j\mid i,j \in \bbN\} \longrightarrow  \spn\{ q_i\B q_j\mid i,j \in \bbN\},
\]
satisfying
\[ 
\hphi (a)= \sum_{i,j =1}^{n} \phi_{ij}(p_iap_j), 
\]
provided that 
\begin{equation} \label{eq;a}
a = \sum_{i,j =1}^{n}  p_iap_j \in \A, \,\, n \in \bbN. 
\end{equation}
\vspace{.1in}

\noindent{\textbf{Claim 2.}} \textit{If $a \in \A$ is as in \textup{(\ref{eq;a})}, then }
\[
\frac{1}{2}\|a\|\leq \|\hphi(a)\|\leq 2\|a\|.
\]
 
\textit{Proof of Claim 2.} Indeed let $\H_n$ be any $n$-dimensional subspace of $\H$ containing the first $n$ vectors vectors $\{f_i\}_{i=1}^n$ of Claim 1. Let $\{h_i\}_{i=1}^n$ be an orthonormal basis of $\H_n$ and let $T \in B(\H_n)$ be a linear operator satisfying $Th_i =f_i$, $i=1, 2, \dots, n$. This is precisely the operator $T$ appearing in Lemma~\ref{lm;basic} and note that $T$ satisfies 
\[
\sqrt{3}/2\|x\|\leq \|Tx\| \leq \sqrt{5}/2\|x\|, \,\, x \in \spn\{h_1, h_2, \dots h_n\},
\]
because of (\ref{eq;bi}) and (\ref{eq;sum}). Now the map
\[
\spn\{ q_i\B q_j\mid 1\leq i,j \leq n \}  \ni \sum_{i,j =1}^{n}  q_ib_{ij}q_j \longmapsto  \sum_{i,j =1}^{n}  q_ib_{ij}q_j \otimes h_{ij}, 
\]
is an injective $*$-homomorphism and so an isometry. Hence 
\begin{equation} \label{eq;long}
\begin{split}
\|\hphi(a)\| &= \|  \sum_{i,j =1}^{n} \phi_{ij}(p_iap_j) \| =  \|  \sum_{i,j =1}^{n} \phi_{ij}(p_iap_j)\otimes h_{ij} \|\\
      &= \|  \sum_{i,j =1}^{n} \phi_{ij}(p_iap_j)\otimes T^{-1}f_{ij}(T^{-1})^* \| \\
       &=  \| (I\otimes T^{-1}) (\sum_{i,j =1}^{n} \phi_{ij}(p_iap_j)\otimes f_{ij} ) (I\otimes (T^{-1})^*)\| \\
       &\leq \|T^{-1}\|^2\big\|\sum_{i,j =1}^{n} \phi(p_iap_j\otimes e_{ij}) \big\|  \\
      & \leq \Big(\frac{2}{\sqrt{3}}\Big)^2 \big\| \phi(\sum_{i,j =1}^{n} p_iap_j\otimes e_{ij})  \big\| \leq 2 \big\| \sum_{i,j =1}^{n} p_iap_j\otimes e_{ij}  \big\| \\
      &=2 \big\|\sum_{i,j =1}^{n} p_iap_j  \big\|  = 2\| a \| .
      \end{split}
\end{equation}
Using arguments similar to those of (\ref{eq;long}), we also have
\begin{equation}
\begin{split}
\|a \| &= \| \phi(\sum_{i,j =1}^{n} p_iap_j\otimes e_{ij})  \big\| \\
       &= \| (I\otimes T) (\sum_{i,j =1}^{n} \phi_{ij}(p_iap_j)\otimes h_{ij} ) (I\otimes T^*)\|  \\
       &\leq (\sqrt{5}/2)^{2} \big\| \sum_{i,j =1}^{n} \phi_{ij}(p_iap_j)\otimes h_{ij} \big\| \leq 2 \|\hphi (a)\|,
      \end{split}
\end{equation}
and so the proof of Claim 2 is complete.
 \vspace{.1in}
 
Claim 2 implies that $\hphi$ extends to a bicontinuous liner map from $\A$ onto $\B$. It remains to verify multiplicativity.
If $a,a' \in \A$ and $1\leq i, j, k \leq n$, then 
\begin{equation} \label{eq;lastcase}
\begin{split}
\hphi(p_iap_j)\hphi(p_ja'p_k)\otimes f_{ik} &=\big(\hphi(p_iap_j)\otimes f_{ij} \big)\big(\hphi(p_j a' p_k)\otimes f_{jk} \big) \\
          &=\big(\phi_{ij}(p_iap_j)\otimes f_{ij} \big)\big(\phi_{jk}(p_j a' p_k)\otimes f_{jk} \big) \\
                   &= \phi(p_iap_j\otimes e_{ij}) \phi(p_sa'p_t\otimes e_{jk} )  \\
          &=\phi(p_iap_ja'p_k\otimes e_{ik})\\
          &=\hphi (p_iap_jp_ja'p_k )\otimes f_{ik},
\end{split}
\end{equation}

On the other hand, if $a,a' \in \A$ and $1\leq i,j,k,l \leq n$ with $j\neq k$, then 
\[
\hphi(p_i a p_j) \hphi(p_k a' p_l) = \phi_{ij}(p_i a p_j) \phi_{kl}(p_k a'p_l) = 0 = \hphi(p_i a p_j p_k a' p_l),
\]
since the ranges of $\phi_{ij}$ and  $\phi_{kl}$, i.e., $q_i \B q_j$ and  $q_k \B q_l$ respectively, are orthogonal whenever $j \neq k$. Hence we have multiplicativity in that case as well and the proof of the theorem in the case where $\dg \A$ is infinite dimensional is complete.
 
Assume now that $\dg \A$ is finite dimensional and so 
\[
\dim(\dg\A ) = \dim ( \dg\B) = n \in \bbN
\]
 by (\ref{eq;fininf}). In that case, choose any orthonormal set $\{e_i\}_{i=1}^n$ in $\H$ and let $\{f_i\}_{i=1}^n$ be unit vectors so that $\phi(p_i\otimes e_{ii})=q_i\otimes f_{ii}$, $1\leq i \leq n$. Arguing as in the previous case,
  $$\phi(p_i\A p_j\otimes e_{ij})  = q_i\B q_j\otimes f_{ij}$$ and so for each pair $1\leq i, j \leq n$ we have surjective linear isometries $\phi_{ij}\colon p_i\A p_j\rightarrow q_i\B q_j$
satisfying $\phi(p_iap_j \otimes e_{ij}) =\phi_{ij}(p_iap_j ) \otimes f_{ij}$, for any $a \in \A$. So by putting all the $\phi_{ij}$ together, we obtain once again a bicontinuous map 
\[
\hphi \colon \A = \spn\{ p_i\A p_j\mid 1\leq i, j \leq n\} \longrightarrow  \spn\{ q_i\B q_j\mid 1\leq i, j \leq n \} = \B,
\]
satisfying
\[ 
\hphi (a)= \sum_{i,j =1}^{n} \phi_{ij}(p_iap_j) \mbox{ with } a = \sum_{i,j =1}^{n}  p_iap_j \in \A. 
\]
\vspace{.1in}
Arguing as in (\ref{eq;lastcase}) one can verify that $\hphi$ is multiplicative and the conclusion follows.
 \end{proof}
 
 As an illustrative application of the previous result, we obtain the following variation of Corollary~\ref{cor;graph}.
 
 \begin{corollary} \label{cor;graph2}
Let $\G$ and $\G'$ be countable directed graphs and let $\T_n$ denote the Cuntz-Toeplitz $\ca$-algebra generated by $n$ isometries with orthogonal ranges, $1\leq n < \infty$ . The following are equivalent:
\begin{itemize}
\item[\textup{(i)}] $\G$ and $\G'$ are isomorphic graphs;
\item[\textup{(ii)}] $\T^+(\G)$ and $\T^+(\G')$ are bicontinuously isomorphic;
\item[\textup{(iii)}] $\T^+(\G)\otimes \T_n$ and $\T^+(\G')\otimes \T_n$ are (completely) isometrically isomorphic.
\end{itemize} 
\end{corollary}

\begin{proof}
The proof is identical to that of Corollary~\ref{cor;graph} by replacing $\K$ with $\T_n$ and using Theorem~\ref{thm;main2} instead of Theorem~\ref{thm;main1}.
\end{proof}

Does the statement of the above Corollary remain true if we replace $\T_n$ with the Cuntz $\ca$-algebra $\O_n$, $1\leq n \leq \infty$?

\section{Algebras with richer diagonals}
We begin by reviewing some basic facts from the $K$-theory of $\ca$-algebras. We adopt the terminology and notation of \cite{Rordam}, which we also use as reference, together with \cite{Black}. 

We say that a $\ca$-algebra $\C$ has cancellation if the semigroup $\D(\widetilde{\C})$ (see Definition 2.3.3 in \cite{Rordam}) has cancellation. As Blackadar comments in \cite[V.2.4.13]{Black}, cancellation for $\C$ implies cancellation for $\C\otimes \K$. Hence both unitizations $\widetilde{\C}$ and $\widetilde{\C\otimes \K}$ are stably finite and so their $K_0$-groups, equipped with their positive cones, form ordered abelian groups \cite[Proposition 5.1.5]{Rordam}.

Assume that the $\ca$-algebras $\C_{i}$ have cancellation and furthermore\break 
$K_0(\C_i)\simeq \bbZ$, $i=1,2$. Since $K_0^+(\C_i)\cap (- K_0^+(\C_i)) = \{0\}$, we may assume that $K_0^+(\C_i)\subseteq \bbZ^+$. Therefore any $*$-isomorphism $\phi: \C_1\rightarrow\C_2$ induces a positive isomorphism $K_0(\phi)$ of $K_0$-groups and so it is the identity map on $\bbZ$. Finally recall that if $e \in \K$ is a minimal projection, then the map $\C\ni c\mapsto c\otimes e \in \C\otimes \K$ induces a positive isomorphism at the $K_0$-level and therefore the identity map on $\bbZ$. In particular, if $\C$ is unital, $[1]_{K_0(\C)}=[1 \otimes e]_{K_0(\C \otimes \K)}$.

\begin{theorem} \label{thm;smain}
Let $\A$ and $\B$ be unital operator algebras. Assume that both $\dg \A$ and $\dg\B$ have cancellation and furthermore assume that $$K_0(\dg\A)\simeq  K_0(\dg\B) \simeq \bbZ$$ with $[1]_{K_0(\dg\A)}=[1]_{K_0(\dg\B)}$. Then $\A\otimes\K$ and $\B\otimes\K$ are completely isometrically isomorphic if and only if $\A$ and $\B$ are completely isometrically isomorphic. 
\end{theorem}

\begin{proof}
Assume that we have a completely isometric isomorphism 
\[
\phi: \A\otimes\K \longrightarrow \B\otimes\K.
\]
By Lemma \ref{wk}, the restriction of $\phi$ on $\dg(\A\otimes \K)= (\dg\A)\otimes \K$ becomes an isomorphism onto $\dg(\B \otimes \K)$, which is the identity map on $\bbZ$ at the $K_0$-level, in accordance to our earlier identifications. Let $e\in \K$ be a minimal projection. Then as elements of $\bbZ$,
\begin{align*}
[1\otimes e]_{K_0((\dg\B) \otimes \K)}&=[1]_{K_0(\dg\B)}= [1]_{K_0(\dg \A))} \\
		&=[1 \otimes e]_{K_0((\dg\A)\otimes \K)} \\
		&=[\phi ( 1 \otimes e)]_{K_0((\dg\B)\otimes \K)} 
\end{align*}
Therefore $\phi(1\otimes e) \in  [1\otimes e]_{K_0((\dg\B)\otimes \K)}$ and since $(\dg\B)\otimes \K$ has cancellation, $1\otimes e$ and $\phi(1\otimes e)$ are Murray-von Neumann equivalent projections. By \cite[Proposition 2.2.8]{Rordam}, there exists a unitary $u \in M_2(\widetilde{\dg(\B\otimes\K)})$ so that 
\[
  u \begin{pmatrix}
  \phi( 1\otimes e) & 0 \\
   0 & 0\\
  \end{pmatrix} u^* =  \begin{pmatrix}
  1\otimes e & 0 \\
   0 & 0\\
  \end{pmatrix} .
  \]
Hence by conjugating the completely isometric isomorphism 
 \[
 \phi^{(2)}:M_2( \A\otimes\K) \longrightarrow M_2( \B\otimes\K)
 \]
 by $u$, we obtain a completely isometric isomorphism
 \[
 \psi: M_2( \A\otimes\K) \longrightarrow M_2( \B\otimes\K)
 \]
 with $\psi( \begin{psmallmatrix}  1\otimes e & 0 \\ 0 & 0 \end{psmallmatrix}) = \begin{psmallmatrix}  1\otimes e & 0 \\ 0 & 0 \end{psmallmatrix}$. Combined with the fact that 
 \[ 
 \begin{psmallmatrix}  1\otimes e & 0 \\ 0 & 0
\end{psmallmatrix}
   M_2( \A\otimes\K)
 \begin{psmallmatrix}  1\otimes e & 0 \\ 0 & 0 \end{psmallmatrix}  =
 \begin{psmallmatrix}  \A\otimes \bbC  e & 0 \\ 0 & 0 \end{psmallmatrix} \simeq \A,
  \]
the conclusion follows by restricting $\psi$ to the algebra above.  
  \end{proof}
  
  \begin{remark}
  The condition $[1]_{K_0(\dg(\A))}=[1]_{K_0(\dg(\B))}$ cannot be dropped from the statement of Theorem~\ref{thm;smain}. Indeed, $M_m(\bbC)$ and $M_n(\bbC)$, $m \neq n $, have cancellation and are stably isomorphic but they are not isomorphic. 
  \end{remark}
  
A fundamental class of $\ca$-algebras $\C$ with cancellation, that satisfy $K_0(\C)\simeq \bbZ$, consists of all $\ca$-algebras of the form $\C \simeq C(X)$, with $X$ a contractible compact space. For operator algebras with such diagonals we can offer a much stronger version of Theorem~\ref{thm;smain}.

\begin{lemma} \label{path}
Let $X $ be a contractible, compact Hausdorff space and let $\fK$ be any $\ca$-algebra. Then any projection-valued function in $C(X)\otimes \fK$ is equivalent to a constant function via a unitary in the multiplier algebra.
\end{lemma}

\begin{proof}
Let $f \in C(X, \fK)\simeq C(X)\otimes \fK$ be such a projection-valued function. We claim that there exists a point $x_0 \in X$ and a unitary $u \in C(X, B(\H))$ such that $u(x)f(x)u^*(x)=f(x_0)$, for all $x \in X$.

Since $X$ is contractible, there exists a continuous map $h:X\times [0,1]\rightarrow X$ and a point $x_0 \in X$ so that $h(x,0)=x$ and $h(x,1)=x_0$, for all $x \in X$. Let $f_s(x):= f\circ h(x, s)$, $x \in X$. Then the map 
\[
[0,1]\ni s \longmapsto f_s\in C(X, \fK)
\]
establishes a homotopy between $f$ and the constant projection valued function $f_1 \in C(X, \fK)$. The conclusion now follows from \cite[Proposition 2.2.6]{Rordam}.
\end{proof}

 \begin{theorem} \label{thm;contract}
Let $\A$ and $\B$ be analytic operator algebras whose diagonals are isomorphic to the continuous functions on (perhaps distinct) compact, contractible Hausdorff spaces. Let $\fK_{\A}$ and $\fK_{\B}$ be operator algebras containing the compact operators. If $\A \otimes\fK_{\A}$ and $\B  \otimes\fK_{\B}$ are isometrically isomorphic, then $\A$ and $\B$ are isometrically isomorphic. 
 \end{theorem}
 
 \begin{proof}
 Assume that $\phi: \A \otimes\fK_{\A} \rightarrow \B \otimes\fK_{\B}$ is an isometric isomorphism and let $e \in \K$ be a rank-one projection. Since isometric isomorphisms map diagonals to diagonals, Lemma~\ref{path} implies that perhaps after conjugating $\phi$ with a unitary, we have $$\phi(1\otimes e) = 1 \otimes f,$$ for some projection $f \in \fK_{\B}$. Notice now that by Lemma~\ref{l;dg0} 
 \begin{align*}
 (1\otimes e)\dg(\A\otimes \fK_{\A})((1\otimes e) &= (\dg\A) \otimes  e (\dg\fK_{\A}) e \\
 						&=(\dg\A) \otimes \bbC e ,
 \end{align*}
which is abelian. Hence 
\begin{align*}
(\dg\B) \otimes ( f \K f ) &\subseteq (1\otimes f)\dg(\B\otimes \fK_{\B})(1\otimes f) \\
&= \phi\big( (1\otimes e)\dg(\A\otimes \fK_{\A})(1\otimes e) \big)
\end{align*}
which is the image of an abelian algebra. This forces $f$ to have rank one and so the restriction of $\phi$ on 
\[(1\otimes e)(\A\otimes \fK_{\A})(1\otimes e) \simeq \A\otimes \bbC e \simeq \A
\]
is an isomorphism onto $(1\otimes f)(\B\otimes \fK_{\B})(1\otimes f)  \simeq \B$ and the conclusion follows.
 \end{proof}
 
 As an immediate application we have the following result that provides a classification scheme for a large class of non-selfadjoint crossed products.

\begin{corollary} \label{motivation2}
Let $X $ be a contractible, compact Hausdorff space and let $\sigma_{i} : X\rightarrow X$, $i=1,2$, be homeomorphisms. Then the following are equivalent:
\itemize
\item[\textup{(i)}] the algebras $C(X)\rtimes_{\sigma_i}\bbZ^+$, $i=1,2$, are {(completely)} isometrically isomorphic;
\item[\textup{(ii)}] the algebras $(C(X)\rtimes_{\sigma_i}\bbZ^+)\otimes \K$, $i=1,2$, are (completely) isometrically isomorphic;
\item[\textup{(iii)}] the algebras $(C(X)\otimes \K^+) \rtimes_{\sigma_i \otimes \Ad \lambda}\bbZ$, $i=1,2$, are (completely) isometrically isomorphic;
\item[\textup{(iv)}]  $\sigma_1$ is conjugate to $\sigma_2$.
\end{corollary}

\begin{proof}
The equivalence of (i) and (iv) follows from \cite[Corollary 3.7]{HH}. The equivalence of (ii) and (iii) was discussed in the introduction of this paper. Theorem \ref{thm;contract} shows that (ii) implies (i). Finally, if (iv) holds then the crossed product $\ca$ algebras $C(X)\rtimes_{\sigma_i}\bbZ$, $i=1,2$, are canonically isomorphic and from this (i) follows.
\end{proof}

Let $\A$ be a $\ca$-algebra and $\alpha$ a $*$-automorphism of $\A$. Let \[
 i \colon \A\otimes \K \longrightarrow (\A\rtimes_{\alpha}\bbZ)\otimes \K
 \]
 be the embedding of $\A\otimes \K$ in $(\A\rtimes_{\alpha}\bbZ)\otimes \K$ induced by the canonical map $\A \hookrightarrow \A\rtimes_{\alpha}\bbZ$, that maps elementary tensors to elementary tensors and let  $U$ denote the universal unitary in $\A\rtimes_{\alpha}\bbZ$. It is easy to see that the pair $(i, U\otimes I)$ forms a covariant representation for the dynamical system $( \A \otimes \K , \alpha\otimes \id)$ that admits a gauge action which fixes $i( \A\otimes \K )$ and ``twists" $U\otimes I$. Hence $(i, U\otimes I)$ integrates to a faithful representation of $(\A \otimes \K)\rtimes_{\alpha\otimes \id}\bbZ $ and therefore we obtain  a canonical isomorphism 
 \begin{equation} \label{eq;cp}
 \sigma_{\alpha} :  (\A\rtimes_{\alpha}\bbZ)\otimes \K  \longrightarrow (\A \otimes \K)\rtimes_{\alpha\otimes \id}\bbZ
 \end{equation}
 satisfying $\sigma_{\alpha}( a U^n \otimes K) = (a\otimes K)(U\otimes I)^n$, $a \in \A$, $K\in \K$, $n \in \bbN$. In  particular $\sigma_{\alpha}$ acts as the identity map on $\A\otimes \K$.
This shows that the study of stable isomorphisms between crossed or semicrossed products reduces to the study of the usual  isomorphism problem between such algebras.
 
 For semicrossed products of \textit{unital} $\ca$-algebras by $*$-automorphisms, Davidson and the first-named author \cite{DavKak} have completely solved the (isometric) isomorphism problem. The more general isomorphism problem for semicrossed products of \textit{unital} $\ca$-algebras by $*$-endomorphisms was  recently resolved by the second-named author and Ramsey in \cite{KR22}. Unfortunately, the arguments in these papers seem to depend heavily on the unitality of the $\ca$-algebras involved  and so these works cannot be combined with the reduction of the previous paragraph in order to solve the stable isomorphism problem, since algebras of the form $\A \otimes \K$ are never unital. We therefore conclude the paper with asking for a solution of the isomorphism problem for the semicrossed products of non-unital $\ca$-algebras.

 \begin{acknow}
The authors thank the anonymous referees for their constructive comments.

Evgenios Kakariadis acknowledges support from EPSRC as part of the programme ``Operator Algebras for Product Systems'' (EP/T02576X/1). 
Elias Katsoulis was partially supported by the NSF grant DMS-2054781.
Xin Li has received funding from the European Research Council (ERC) under the European Union’s Horizon 2020 research and innovation programme (grant agreement No. 817597).
\end{acknow}

\noindent 
{\bf Data availability statement.}
For the purposes of publication of this article, we note that data sharing is not applicable as no datasets were generated or analysed during the underlying research.

\smallskip

\noindent 
{\bf Conflict of interest statement.}
On behalf of all authors, the corresponding author states that there is no conflict of interest.

\smallskip

\noindent 
{\bf Open access statement.}
For the purpose of open access, the second author has applied a Creative Commons Attribution (CC BY) license to any Author Accepted Manuscript version arising.


\end{document}